\documentclass[11pt]{amsart}
\usepackage{amsfonts,amssymb,amscd,amsmath,enumerate,verbatim,calc,latexsym}
\input xy
\xyoption{all}

\def\NZQ{\mathbb}               

\def\RR{{\NZQ R}}

\def\frk{\mathfrak}               

\def\mm{{\frk m}}

\def\Phi{{\frk N}}
\def\opn#1#2{\def#1{\operatorname{#2}}} 
\opn\chara{char} \opn\length{\ell} \opn\pd{pd} \opn\rk{rk}
\opn\projdim{proj\,dim} \opn\injdim{inj\,dim} \opn\rank{rank}
\opn\depth{depth} \opn\grade{grade} \opn\height{height}
\opn\embdim{emb\,dim} \opn\codim{codim}

\opn\Tr{Tr} \opn\bigrank{big\,rank}
\opn\superheight{superheight}\opn\lcm{lcm}
\opn\trdeg{tr\,deg}
\opn\reg{reg} \opn\lreg{lreg} \opn\ini{in} \opn\lpd{lpd}
\opn\size{size}\opn{\mult}{mult}
\opn\div{div} \opn\Div{Div} \opn\cl{cl} \opn\Cl{Cl}
\opn\Spec{Spec} \opn\Supp{Supp} \opn\supp{supp} \opn\Sing{Sing}
\opn\Ass{Ass} \opn\Min{Min}
\opn\Ann{Ann} \opn\Rad{Rad} \opn\Soc{Soc}
\opn\Syz{Syz} \opn\Im{Im} \opn\Ker{Ker} \opn\Coker{Coker}
\opn\Am{Am} \opn\Hom{Hom} \opn\Tor{Tor} \opn\Ext{Ext}
\opn\End{End} \opn\Aut{Aut} \opn\id{id}

\opn\nat{nat}
\opn\pff{pf}
\opn\Pf{Pf} \opn\GL{GL} \opn\SL{SL} \opn\mod{mod} \opn\ord{ord}
\opn\Gin{Gin}
\opn\Hilb{Hilb}\opn\adeg{adeg}\opn\std{std}\opn\ip{infpt}
\opn\Pol{Pol}
\opn\sat{sat}
\opn\Var{Var}

\opn\aff{aff} \opn\con{conv} \opn\relint{relint} \opn\st{st}
\opn\lk{lk} \opn\cn{cn} \opn\core{core} \opn\vol{vol}
\opn\link{link} \opn\star{star}
\opn\gr{gr}

\def\pot#1#2{#1[\kern-0.28ex[#2]\kern-0.28ex]}

\opn\dirlim{\underrightarrow{\lim}}
\opn\inivlim{\underleftarrow{\lim}}
\let\union=\cup
\let\sect=\cap

\let\Union=\bigcup

\let\To=\longrightarrow
\def\Implies{\ifmmode\Longrightarrow \else
        \unskip${}\Longrightarrow{}$\ignorespaces\fi}
\def\implies{\ifmmode\Rightarrow \else
        \unskip${}\Rightarrow{}$\ignorespaces\fi}
\def\iff{\ifmmode\Longleftrightarrow \else
        \unskip${}\Longleftrightarrow{}$\ignorespaces\fi}

\let\:=\colon
\newtheorem{Theorem}{Theorem}[section]
\newtheorem{Lemma}[Theorem]{Lemma}
\newtheorem{Corollary}[Theorem]{Corollary}
\newtheorem{Proposition}[Theorem]{Proposition}

\newtheorem{Definition}[Theorem]{Definition}

\let\epsilon\varepsilon
\let\phi=\varphi
\let\kappa=\varkappa
\def\qed{\ifhmode\textqed\fi
      \ifmmode\ifinner\quad\qedsymbol\else\dispqed\fi\fi}
\def\textqed{\unskip\nobreak\penalty50
       \hskip2em\hbox{}\nobreak\hfil\qedsymbol
       \parfillskip=0pt \finalhyphendemerits=0}
\def\dispqed{\rlap{\qquad\qedsymbol}}

\opn\dis{dis}
\def\pnt{{\raise0.5mm\hbox{\large\bf.}}}

\opn\Lex{Lex}

\begin{document}

\title{Cohen-Macaulay monomial ideals of  codimension 2}

\author{Muhammad Naeem}
\thanks{}
\subjclass{13C14, 13D02, 13D25, 13P10}
\address{Muhammad Naeem, Abdus Salam School of Mathematical Sciences (ASSMS), GC University, Lahore, Pakistan.
}
\email{naeem\_pkn@yahoo.com}

\begin{abstract}
We give a structure theorem for Cohen-Macaulay monomial ideals of
codimension 2, and describe all possible relation matrices of such
ideals. In case that the ideal has a linear resolution, the relation
matrices can be identified with the spanning trees of a connected
chordal graph with the property that each distinct pair of maximal
cliques of the graph  has  at most one vertex in common.

\vskip 0.4 true cm
 \noindent
  {\it Key words } : Monomial Ideals, Taylor Complexes, Linear Resolutions, Chordal Graphs.
\end{abstract}

\maketitle
\section*{Introduction}
The purpose of the paper is to work out in detail a remark on the
structure of Cohen-Macaulay monomial ideals of codimension 2 which
was made in the paper \cite{BH1}. There it was observed that the
`generic' ideals of this type, generated by $n$  elements,  are in
bijective correspondence to the trees with $n$  vertices. In
Proposition~\ref{lahore} we give an explicit description of the
generators of a generic Cohen-Macaulay monomial ideal of codimension
2 in terms of the associated tree and describe the minimal prime
ideals of such ideals in Proposition~\ref{islamabad}. As a
consequence of these two results we obtain as the main result  of
Section ~1  a full description of all Cohen-Macaulay monomial ideals
of  codimension 2, see Theorem~\ref{juergen}.

In Section~2 we study the possible relation trees of  a
Cohen-Macaulay monomial ideals of  codimension 2. This set of
relation trees is always the set of bases of a matriod
(Proposition~\ref{matroid}), which in case of a generic ideal
consists of only one tree as shown in Proposition~\ref{first}. We
call the graph $G$ whose set of edges is the union of the set of
edges of all relation trees  of a given Cohen-Macaulay monomial
ideal $I$ of  codimension 2, the Taylor graph of $I$. Then each of
the relation trees is a spanning tree of the Taylor graph. The
natural question arises whether the set of relation trees of $I$ is
precisely the set of spanning trees of $G$. We show by an example
that this is not the case in general. On the other hand, we prove in
Theorem~\ref{linear} that each relation tree of $I$ is a spanning
tree of $G$, if $I$ has a linear resolution. In order to obtain a
complete description of all possible relation trees when $I$ has a
linear resolution, it is therefore required to find all possible
Taylor graphs of such ideals. This is done in Theorem~\ref{chordal},
where it is shown that a finite connected simple graph is the Taylor
graph of Cohen-Macaulay monomial ideal of  codimension $2$ with
linear resolution, if and only if $G$ is chordal and any two maximal
cliques of $G$ have at most one vertex in common.

\section{On the structure of Cohen-Macaulay monomial ideals of  codimension 2}

In  \cite[Remark 6.3]{BH1} the following observation was made regarding the structure of a codimension 2  Cohen-Macaulay monomial ideal $I$: let
$$\{u_1,u_2,...,u_{m+1}\}$$ the unique minimal set of monomial  generators of
$I$. Consider the Taylor complex of the sequence
$u_1,u_2,...,u_{m+1}$
\[
...\rightarrow\bigoplus_{i=1}^{m+1\choose 2}S e_i\wedge
e_j\overset{\phi_2}{\rightarrow}\bigoplus_{i=1}^{m+1}S
e_i\overset{\phi_1}{\rightarrow} S
\]
 The matrix corresponding to
$\phi_2$ is of size ${m+1\choose 2}\times m+1$ whose  rows correspond to
Taylor relation  (cf.\ \cite{Ei}), namely to the relations
 \[
 e_i\wedge e_j\mapsto u_{ji}e_j-u_{ij}e_i
 \]
where $i<j$ and $u_{ji}=u_i/\gcd(u_i,u_j)$,
$u_{ij}=u_j/\gcd(u_i,u_j)$.

Let $U=\Ker(\phi_1)$; then the Taylor relations form a homogeneous
system of generators of $U$. Since $\projdim S/I=2$, it follows that
$U$ is free of rank $m$. In particular $U$ is minimally generated by
$m$ elements.  Applying the graded Nakayama Lemma (cf.\ \cite{BH} or
\cite[Lemma 1.2.6]{V}), a minimal system of graded generators of $U$
can be chosen among the Taylor relations. We then obtain a minimal
graded free resolution
\[
0\rightarrow S^{m}\overset{A}{\rightarrow} S^{m+1} \rightarrow S
\rightarrow S/I \rightarrow 0
\]
of $S/I$, where $A$ is a matrix whose rows correspond to Taylor
relations. Any such matrix will be called a {\em Hilbert--Burch matrix} of $I$

Notice that  each row of $A$ has exactly two nonzero
entries.
We obtain a graph $\Gamma$ on the
vertex set $[m+1]=\{1,\ldots,m+1\}$ from the matrix $A$ as follows: we say that
$\{i,j\}$ is an edge of $\Gamma$, if and only if there is a row of
$A$ whose nonzero entries are the $i$th and $j$th components.

We claim that every column of $A$ has a nonzero entry. In fact, if
this would not be the case, say, the $k$th column of $A$ has all
entries zero, then the relation $u_{k+1,k}e_{k+1}-u_{k,k+1}e_k\in U$
could not be written as a linear combination of the minimal graded
homogeneous generators of $U$.  This shows that $\Gamma$ has no
isolated vertex. On the other hand, since the number of vertices of
$\Gamma$ is $m+1$ and the number of edges of $\Gamma$ is $m$, we see
that $\Gamma$ is a tree, which  is called a {\em relation tree} of $I$. The set of all relation trees of $I$ will be denoted by  $\mathcal{T}(I)$.

Conversely, given a tree $\Gamma$ on the vertex set $[m+1]$ with $m\geq 2$, we are going
to construct a codimension 2 Cohen-Macaulay monomial ideal $I$ for which $\Gamma$ is a relation tree.
We assign to $\Gamma$ an $m\times (m+1)$-matrix
$A(\Gamma)=(a_{ij})$ whose entries are either $0$ or indeterminates.
The matrix $A(\Gamma)$ is defined as follows: let $E(\Gamma)$ be the
set of edges of $\Gamma$. Since $\Gamma$ is a tree, there are
exactly $m$ edges. We choose an arbitrary order of the edges of
$\Gamma$, and assign to the $k$th edge $\{i,j\}\in E(\Gamma)$ the
$k$th row of $A(\Gamma)$ by
\begin{eqnarray}
\label{genericmatrix}
a_{kl}= \left\{\begin{array}{ll} -x_{ij}&\text{ if $l =i$,}\\
x_{ji}&\mbox{ if $l=j$,}\\
0&\mbox{ otherwise.}
\end{array}\right.
\end{eqnarray}
For example if $\Gamma$ is the tree with edges $\{1,2\}$, $\{2,3\}$
and $\{2,4\}$. Then we obtain the matrix
\[
A(\Gamma)=
\begin{pmatrix}
 -x_{12} & x_{21} & 0 & 0 \\
0 &- x_{23} & x_{32} & 0 \\
0 & -x_{24} & 0 & x_{42}
\end{pmatrix}
\]

\begin{Definition}{\em Let $\Gamma$  be a tree on the vertex set $[m+1]$ and $i,j$
be two distinct vertices of $\Gamma$. Then there exists a unique
path from $i$ to $j$ denoted by $i \rightarrow j$, in other words a
sequence of numbers $i=i_0,i_1,i_2,\ldots,i_{k-1},i_k=j$ such that
$\{i_l,i_{l+1}\}\in E(\Gamma)$ for $l=0,\ldots,k-1$.
We set \[
b(i,j)=i_{1}\quad\text{and}\quad
e(i,j)=i_{k-1}\]}
\end{Definition}

\begin{Proposition}
\label{lahore}
Let  $v_{j}$ be the minor  of $A(\Gamma)$ which is obtained
 by omitting the $j$th
column of $A(\Gamma)$. Then  $v_j=\pm\prod_{i=1\atop i\neq
j}^{m+1}x_{ib(i,j)}$ for $j=1,2,...,m+1$
\end{Proposition}

\begin{proof}
We prove the assertion  by using induction on the number of edges of
$\Gamma$.
 If $|E(\Gamma)|=~1$, then
  \[
  A(\Gamma)=(-x_{12},x_{21})
\]
Therefore, $v_{1}=x_{21}$ and $v_2=-x_{12}$, as required.

Now assume that the assertion is true for $|E(\Gamma)|=m-1\geq 1$.
Since $\Gamma$ is a tree, there exists a free
 vertex of $\Gamma$, that is, a vertex which belongs to exactly one edge. Such an
 edge of $\Gamma$ is called a leaf.  We may assume the edge $\{m,m+1\}$ is a leaf
 and that $m+1$ is a free vertex of $\Gamma$. The tree which is obtained from
 $\Gamma$ by removing the leaf $\{m,m+1\}$ will be denoted by $\Gamma'$.
  By our induction hypothesis the minors $v_1',\ldots, v_m'$ of $\Gamma'$ have the
  desired form.
We may assume that the edge $\{m,m+1\}$ is the last in the order of
edges. Then $(m-1)\times m$ matrix $A(\Gamma')$ is obtained from the
$m\times (m+1)$-matrix $A(\Gamma)$ by removing the last row
\[
R_{m}= (0, \ldots, 0,-x_{m,m+1}, x_{m+1,m})
\]
and the last column
\[
\begin{pmatrix} 0 \\ \vdots \\0 \\x_{m+1,m}
\end{pmatrix}
\]
It follows that  the minors $v_1,\ldots, v_{m+1}$  of $A(\Gamma)$
are given by
\begin{eqnarray}
\label{recursion}
v_j=x_{m+1,m}v_j'\quad \text{for} \quad j=1,\ldots, m, \quad
\text{and} \quad v_{m+1}=x_{m,m+1}v_m'.
\end{eqnarray}
Therefore, our induction hypothesis implies that
\[
v_j=x_{m+1,m}v_j'=\pm x_{m+1,m} \prod_{i=1,\;  i\neq
j}^{m}x_{i,b(i,j)}=\pm \prod_{i=1, \; i\neq j}^{m+1}x_{i,b(i,j)}
\]
for $j=1,\ldots,m$, and
\[
v_{m+1}=x_{m,m+1}v_m'=\pm x_{m,
m+1}\prod_{i=1}^{m-1}x_{i,b(i,m)}=\pm x_{m,b(m,m+1)}\prod_{i\neq
i=1}^{m-1}x_{i, b(i,m+1)},
\]
 because $b(i,m)=b(i,m+1)$ for all
$i\leq m$. So this implies that
\[
v_{m+1}=\pm \prod_{i=1,\; i\neq m+1}^{m+1}x_{i,b(i , m+1)},
\]
as desired.
\end{proof}

For a tree $\Gamma$ on the vertex set $[m+1]$ we denote by
$I(\Gamma)$ the ideal generated by the minors $v_1,\ldots, v_{m+1}$
of $A(\Gamma)$ and call it the {\em generic}   monomial ideal attached to the tree $\Gamma$.

\begin{Corollary}
\label{genmon} The ideal $I(\Gamma)$ is a Cohen--Macaulay ideal of
codimension $2$.
\end{Corollary}

\begin{proof}
The greatest common divisors of the monomial generators $v_j$  of
$I(\Gamma)$ is one. This can easily be seen by the formulas
(\ref{recursion}) in the proof of Proposition~\ref{lahore}. The
assertion follows then from \cite[Theorem 1.4.17]{BH}.
\end{proof}

The generic ideal $I(\Gamma)$  has the following nice primary decomposition:

\begin{Proposition}
\label{islamabad}
 $I(\Gamma)=\bigcap_{1\leq i<
j\leq m+1}(x_{ib(i,j)},x_{je(i,j)})$.
\end{Proposition}

\begin{proof}
 We prove the assertion by using  induction on the number of edges of $\Gamma$. For
$|E(\Gamma)|=1$ we have,
\[
A(\Gamma)=( -x_{12},x_{21}).
\]
with $v_1=x_{21}$ , $v_2=-x_{12}$. Therefore
$I(\Gamma)=(x_{21},x_{12})=(x_{1b(1,2)},x_{2e(1,2)})$. Now assume
that assertion is true if  $|E(\Gamma)|=m-1\geq 1$. Since $\Gamma$
is a tree, there exists a free vertex of $\Gamma$, that is, a vertex
which belongs to exactly one edge. Such an edge of $\Gamma$ is
called a leaf.  We may assume the $\{m,m+1\}$ is a leaf and that
$m+1$ is a free vertex of $\Gamma$. The tree which is obtained from
$\Gamma$ by removing the leaf $\{m,m+1\}$ will be denoted by
$\Gamma'$. So then for $A(\Gamma')$ we have
\[
I(\Gamma')=(v'_1,v'_2,...,v'_m)=\bigcap_{1\leq i<j\leq
m}(x_{ib(i,j)},x_{je(i,j)}).
\]

We may assume that the edge $\{m,m+1\}$ is the last in the order of
edges. Then $(m-1)\times m$ matrix $A(\Gamma')$ is obtained from the
$m\times (m+1)$-matrix $A(\Gamma)$ by deleting the last row
\[
R_{m}=(0,  \ldots, 0 , -x_{m,m+1}, x_{m+1,m})
\]
and the last column
\[
\begin{pmatrix} 0 \\ \vdots \\0 \\x_{m+1,m}
\end{pmatrix}
\]
It follows that  the minors $v_1,\ldots, v_{m+1}$  of $A(\Gamma)$
are given by
\[
v_j=x_{m+1,m}v_j'\quad \text{for} \quad j=1,\ldots, m, \quad
\text{and} \quad v_{m+1}=x_{m,m+1}v_m'.
\]

Hence $$I(\Gamma)=(v_1,v_2,...,v_{m+1}).$$

On the other hand, by using the induction hypothesis and the fact that  $e(i,m+1)=m$
for all $i\leq m$, we get
\begin{eqnarray*}
\bigcap_{1\leq i<j\leq m+1}(x_{ib(i,j)},x_{je(i,j)})&= &\bigcap_{1\leq i<j\leq m}(x_{ib(i,j)},x_{je(i,j)})\cap \bigcap_{i=1}^{m}(x_{ib(i,m+1)},x_{m+1,e(i,m+1)})\\
&=&(v'_1,v'_2,...,v'_m)\cap\bigcap_{i=1}^{m}(x_{ib(i,m+1)},x_{m+1,m})\\
&=&(v'_1,v'_2,...,v'_m)\cap(\prod_{i=1}^{m}x_{ib(i,m+1)},x_{m+1,m})\\
&=&(v'_1,v'_2,...,v'_m)\cap(x_{m,m+1}\prod_{i=1}^{m-1}x_{ib(i,m+1)},x_{m+1,m})\\
&=&(v'_1,v'_2,...,v'_m)\cap(x_{m,m+1}v'_{m},x_{m+1,m}).
\end{eqnarray*}
Observing that $\gcd(v_i',x_{m+1,m})=1$ it follows that
\begin{eqnarray*}
(v'_1,v'_2,...,v'_m)&\cap&(x_{m,m+1}v'_{m},x_{m+1,m})\\
&=&(x_{m+1,m}v'_1,x_{m+1,m}v'_2,...,x_{m+1,m}v'_m,x_{m,m+1}v'_{m})\\
&=&(v_1,v_2,...,v_m,v_{m+1})=I(\Gamma).
\end{eqnarray*}
Hence
\[
I(\Gamma)=\bigcap_{1\leq i<j\leq m+1}(x_{ib(i,j)},x_{je(i,j)}),
\]
as desired.
\end{proof}

As an application of Proposition~\ref{lahore},
Corollary~\ref{genmon} and Proposition~\ref{islamabad} we obtain the
following characterization of Cohen-Macaulay monomial ideals of
codimension 2.

\begin{Theorem}
\label{juergen}
{\em (a)} Let $I\subset S=K[x_1,x_2,...,x_n]$ be a Cohen-Macaulay
monomial ideal of codimension $2$ generated by $m+1$ elements.
Then there exists a tree $\Gamma$ with $m+1$ vertices and for each
edge $\{i,j\}$ of $\Gamma$ there  exists a monomials $u_{ij}$ and $u_{ji}$ in $S$ such that
\begin{itemize}
\item[{(i)}] $\gcd(u_{ib(i,j)},u_{je(i,j)})=1$ for all $i<j$, and
\item[{(ii)}] $I=(\prod_{i=2}^{m+1}u_{ib(i,1)},\ldots,\prod_{i=1\atop i\neq
j}^{m+1}u_{ib(i,j)},\ldots,\prod_{i=1}^{m}u_{ib(i,m+1)})$
\end{itemize}
{\em (b)} Conversely, if $\Gamma$ is a tree with $[m+1]$ vertices and for each
$\{i,j\}\in E(\Gamma)$ we are given monomials $u_{ij}$ and $u_{ji}$ in $S$
satisfying {\em (a)(i)}. Then the ideal defined in {\em (a)(ii)} is
Cohen-Macaulay of codimension $2$.
\end{Theorem}

\begin{proof}
(a) (ii) Let $A$ be an $m\times m+1$ matrix of Taylor relations
which generated the relation module of $U$ of $I$, and let $\Gamma$
be the corresponding relation tree. We apply the  Hilbert--Burch
Theorem (\cite[1.4.17]{BH}) according to which the ideal  $I$ is
generated by the maximal minors of $A$. The matrix $A$ is obtained
from $A(\Gamma)$ by the substitution:
\[
 x_{ij}\mapsto u_{ij}.
\]
Therefore statement (ii) follows from Proposition~\ref{lahore}.

Now we shall prove  assertion (i). For this we use  Proposition
\ref{islamabad} which says that
\[
I(\Gamma)=\bigcap_{1\leq i<j\leq m+1}(x_{ib(i,j)},x_{je(i,j)}).
\]
Applying the substitution map introduced in the proof of (ii) we
obtain
\begin{eqnarray}
\label{naeem}
I\subseteq\bigcap_{i<j}(u_{ib(i,j)},u_{je(i,j)}).
\end{eqnarray}
Suppose  $\gcd(u_{ib(i,j)},u_{je(i,j)})\neq 1$ for some $i$ and $j$.
Then it follows from (\ref{naeem}) that  $I$ is contained in a
principal ideal. This is a contradiction, because $\height I=2$.

(b) Let $\Gamma$ be a tree with vertex set $[m+1]$ and $m$ edges.
For each  $\{i,j\}\in E(\Gamma)$ we have  monomials $u_{ij},
u_{ji}\in S$ satisfying condition (a)(i). Let $A$ be the matrix
obtained from $A(\Gamma)$ by the substitutions
 $x_{ij}\mapsto u_{ij}$, and let $I$ be the ideal generated by the maximal minors of $A$.
 It follows from Proposition~\ref{lahore} that $I=(v_1,\ldots, v_{m+1})$ where
 $v_j=\prod_{i=1\atop i\neq
j}^{m+1}u_{ib(i,j)}$.

First we
shall prove that
\[
\gcd(v_1,v_2,...,v_{m+1})=1.
\]
We shall prove this by induction on the number of edges of $\Gamma$. The assertion is trivial if $\Gamma$
has only one edge. Now let $|E(\Gamma)|=m>1$ and assume  that the assertion is true for any tree with $m-1$ edges.

We may assume that $(m,m+1)$ is a leaf of $\Gamma$. Let $\Gamma'$ be the tree obtained from
$\Gamma$ by removing the edge $\{m,m+1\}$.  The matrix  $A(\Gamma')$ is  obtained from $A(\Gamma)$
by removing the row $(0,\ldots,-x_{m,m+1},x_{m+1,m})$ and the column
\[
\begin{pmatrix} 0 \\ \vdots \\0 \\x_{m+1,m}
\end{pmatrix}.
\]

Let $A'$ be the matrix obtained from
$A(\Gamma')$ by the substitutions $x_{ij}\mapsto u_{ij}$, and let $I'=(v_1',\ldots,v_m')$ be
the ideal of maximal minors of $A'$ where, up to sign,  $v_j'$ is the $j$th maximal minor of
$A'$. Expanding the matrix $A$ we see that
\[
v_{j}=\pm v'_ju_{m+1,m}\quad \text{for} \quad j=1,2,\ldots,m\quad \text{and}\quad   v_{m+1}=\pm v'_{m}u_{m,m+1}.
\]

Therefore
\[
\gcd(v_1,v_2,\ldots,v_m,v_{m+1})=\gcd(v'_1u_{m+1,m},v'_2u_{m+1,m},...,v'_mu_{m+1,m},v_{m+1}).
\]
By induction hypothesis we have $\gcd(v'_1 , v'_2 ,..., v'_{m} )=1$, so that
\[
\gcd(v'_1u_{m+1,m}, v'_2u_{m+1,m},...,v'_mu_{m+1,m} )=u_{m+1,m}.
\] Hence
 it is enough to prove that
\[
\gcd(u_{m+1,m}, v_{m+1})=1.
\]
Note that $u_{m+1,m}=u_{m+1,e(i,m+1)}$ for all $i$, and $v_{m+1}= \prod_{i=1}^{m}u_{ib(i,m+1)}$. Therefore
\[
\gcd(u_{m+1,m}, v_{m+1})=
\gcd(u_{m+1,e(i,m+1)}, \prod_{i=1\atop}^{m}u_{ib(i,m+1)})=1,
\]
since by our hypothesis (a)(i) we have
$\gcd(u_{m+1,e(i,m+1)}, u_{ib(i,m+1)})=1$ for all
$i$.

The  Hilbert--Burch Theorem  \cite[1.4.17]{BH} then implies that $I$
is a perfect ideal of codimension $2$, and hence a Cohen--Macaulay
ideal.
\end{proof}

\section{The possible sets  of relation trees attached to  Cohen-Macaulay monomial ideals of codimension 2}

In this section we want to study set $\mathcal{T}(I)$ of all
relation trees of a Cohen--Macaulay monomial ideal of codimension 2.
In general one may have more than just one Hilbert--Burch matrix for
an ideal  $I$, and consequently more than one relation trees.
 For example the ideal
$I=(x_4x_5x_6,x_1x_5x_6,x_1x_2x_6,x_1x_2x_3x_5)\subset
 S=K[x_1,x_2,x_3,x_4,x_5,x_6]$ has the following two Hilbert--Burch matrices
\[
A_1=
\begin{pmatrix}
 -x_1 & x_4 & 0 & 0 \\
0 & -x_2 & x_5 & 0 \\
0 & 0 & -x_3x_5 & x_6
\end{pmatrix},
\]

or

\[
A_2=
\begin{pmatrix}
 -x_1 & x_4 & 0 & 0 \\
0 & -x_2 & x_5 & 0 \\
0 & -x_2x_3 & 0 & x_6
\end{pmatrix}.
\]
The corresponding relation trees are $\Gamma_1$ and $\Gamma_2$ with $E(\Gamma_1)=\{\{1,2\},\{2,3\},\{3,4\}\}$
and $E(\Gamma_2)=\{\{1,2\},\{2,3\},\{2,4\}\}$.

However in the generic case we have

\begin{Proposition}
\label{first} Let $\Gamma$ be a tree on the vertex set $[m+1]$ and
let $I(\Gamma)$ be the generic monomial
ideal attached to $\Gamma$.
Then ${\mathcal T}(I(\Gamma))=\{\Gamma\}$.
\end{Proposition}

Recall that $I(\Gamma)$ is the ideal of maximal minors of the matrix
$A(\Gamma)$ defined in (\ref{genericmatrix}).
Up to signs the minors of $A(\Gamma)$ are the
monomials $v_i=\prod_{r=1\atop r\neq i}^{m+1}x_{rb(r,i)}$, see Proposition~\ref{lahore}.

For the proof of Proposition~\ref{first}  we shall need

\begin{Lemma}
\label{sms} Let $\Gamma$ be a tree, then $\{i,j\}$ is an edge of
 $\Gamma$ if and only if
\[
\lcm(v_{i},v_{j})=v_{j}x_{ji}=v_{i}x_{ij}.
\]
\end{Lemma}

\begin{proof}
Let $\{i,j\}$ be an edge of $\Gamma$ and suppose that $i<j$. Note
that
\begin{eqnarray}
b(k,i)=b(k,j)
\end{eqnarray}
for all $k$ which are different from $i$ and $j$, because if the
path from $k$ to $i$ is $k=k_0,k_1,\ldots,k_l=i$, then the path from
$k$ to $j$ will be $k=k_0,k_1,\ldots,k_{l-1}=j$ or
$k=k_0,k_1,\ldots,k_{l-1},i,j$  since $\{i,j\}$ be an edge of
$\Gamma$. Now using (4) we have
\[
v_i=\pm\prod_{r=1\atop r\neq i}^{m+1}x_{rb(r,i)}
=\pm\prod_{r=1\atop
r\neq i,j}^{m+1}x_{rb(r,i)}x_{jb(j,i)}
=\pm \prod_{r=1\atop r\neq
i,j}^{m+1}x_{rb(r,i)}x_{ji}
=\pm \prod_{r=1\atop r\neq
i,j}^{m+1}x_{rb(r,j)}x_{ji}.
\]
Similarly $v_j=\pm \prod_{r=1\atop r\neq
i,j}^{m+1}x_{rb(r,j)}x_{ij}$. Hence
$\lcm(v_{i},v_{j})=v_{j}x_{ji}=v_{i}x_{ij}$.

On the other hand, suppose  that $\{i,j\}$ is not an edge of
$\Gamma$, then there exists a vertex, different from $i$ and $j$,
say $k$,  which belongs to the path from $i$ to $j$. Therefore
$b(k,i)\neq b(k,j)$, and hence $x_{kb(k,i)}\neq x_{kb(k,j)}$. Since
$x_{kb(k,i)}\mid v_i$ and since $x_{kb(k,j)}\mid v_j$
we cannot have
$\lcm(v_{i},v_{j})=~v_{j}x_{ji}=v_{i}x_{ij}$.
\end{proof}

\begin{proof}[Proof of Proposition \ref{first}]
Since all monomial generators of $I(\Gamma)$ are of degree $m$ and
since, by the Hilbert--Burch Theorem \cite[1.4.17]{BH},  these
generators are the maximal minors of any of its Hilbert--Burch
matrices, it follows that all Hilbert--Burch matrices must be
linear. However by Lemma~\ref{sms}  we have only $m$ linear Taylor
relations. Therefore there exists only one Hilbert--Burch matrix for
$I$.
\end{proof}

In contrast to  the result stated in Proposition~\ref{first} we have

\begin{Proposition}
\label{second} Let $I=(u_1,\ldots,u_{m+1})$ be the monomial ideal in
$K[x_1,\ldots,x_{m+1}]$ with $u_i=x_1\cdots x_{i-1}x_{i+1}\cdots
x_{m+1}$ for $i=1,\ldots,m+1$. Then ${\mathcal T}(I)$ is the set of
all possible trees on the vertex set $[m+1]$.
\end{Proposition}

\begin{proof}
Let $\Gamma$ be an arbitrary tree on the vertex set $[m+1]$. For the
$k$th edge $\{i,j\}$ of $\Gamma$ take the monomial generators $u_i$
and $u_j$ of $I$. Then we have the Taylor relation $x_je_j-x_ie_i$.
Let $A$ be the $m\times m+1$-matrix whose rows
$(0,\cdots,-x_i,\cdots,x_j,\cdots,0)$ correspond to the Taylor
relations $x_je_j-x_ie_i$ arising from the edges of $\Gamma$.
Observe that the generic matrix  $A(\Gamma)$  is mapped to $A$ by
the substitutions $x_{ij}=x_i$. Moreover the maximal minor $\pm v_i$
of $A(\Gamma)$ is mapped to $u_i$ for all $i$. Therefore the $u_i$
are the maximal minors of $A$ which shows that $A$ is the
Hilbert--Burch matrix of $I$.
\end{proof}

In order to study the general nature of $\mathcal{T}(I)$ we
introduce the following concept. Let $\mathcal S$ be a finite set.
Recall that a collection $\mathcal B$ of subsets of $\mathcal  S$ is
said to be the {\em set of bases of a matroid}, if all  $B\in
\mathcal B$ have the same cardinality and if the following exchange
property is satisfied:

\medskip
\noindent For all  $B_1, B_2\in \mathcal B$ and $i\in B_1\setminus B_2$, there exists $j\in B_2\setminus B_1$ such that
$(B_1\setminus\{i\})\union \{j\}\in\mathcal B$.

\medskip
\noindent A classical example  is the
following: let $K$ be a field, $V$ a $K$-vector space and
${\mathcal S}=\{v_1,\ldots, v_r\}$ any finite set of vectors of $V$. Let
$\mathcal B$  the set of subset $B$ of $\mathcal S$ with the property
that $B$ is a maximal set of linearly independent
vectors in $\mathcal S$. It easy to check and well known
that $\mathcal B$ is the set of bases of a matroid.

\begin{Proposition}
\label{matroid} Let $I\subset S$ be a Cohen--Macaulay monomial ideal
of codimension $2$. Then ${\mathcal T}(I)$ is the set of bases of a
matroid.
\end{Proposition}

\begin{proof} Let $I$ be minimally generated by the monomials $u_1,\ldots ,u_{m+1}$ and let

\[
0\To G\To F\To I\To 0
\]
be the graded minimal free $S$-resolution of $S/I$.

The set $\mathcal S$ of  Taylor relations  generate the first syzygy
module $U$ of $I$ which is isomorphic to the free $S$-module $G$.
Consider the graded $K$-vector space $U/\mm U$ where
$\mm=(x_1,\ldots, x_n)$ is the graded maximal ideal of $S$. Note
that $\dim_K U/\mm U=m$. Since the relations $r_{ij}$ generate $U$
it follows that their residue classes $\bar{r}_{ij}$ in the
$K$-vector space  $U/\mm U$ form a system of generators of $U/\mm
U$.  By the homogeneous version of Nakayama (see \cite[1.5.24]{BH})
it follows that a subset $B=\{r_{i_1j_1},\ldots, r_{i_mj_m}\}$ of
the Taylor relations $\mathcal S$ is a minimal set of generators of
$U$ (and hence establishes a Hilbert--Burch matrix of $I$) if and
only if  $\{\bar{r}_{i_1j_1},\ldots, \bar{r}_{i_mj_m}\}$ is a basis
of the $K$-vector space $U/\mm U$. The desired conclusion follows,
since the relation trees of $I$ correspond bijectively to the set of
Hilbert--Burch matrices of $I$.
\end{proof}

Given a finite simple and connected graph $G$. A maximal subtree
$\Gamma\subset G$ is called a {\em spanning tree}. It is well-known
and easy to see that the set ${\mathcal T}(G)$ of spanning trees is
the set of bases of a matroid.

Here we are interested in the spanning trees of  the graph $G(I)$ whose set of edges is given by
with
\[
E(G(I))=\Union_{\Gamma\in \mathcal{T}(I)}E(\Gamma).
\]
We call $G(I)$ the {\em Taylor graph} of $I$. Obviously we have ${\mathcal T}(I)\subset {\mathcal T}(G(I))$. The
question arises whether ${\mathcal T}(I)={\mathcal T}(G(I))$?
Unfortunately this is not always the case as the example at the
beginning of this section shows. Indeed, in this example, ${\mathcal
T}(I)=\{\Gamma_1,\Gamma_2\}$ with
$E(\Gamma_1)=\{\{1,2\},\{2,3\},\{3,4\}\}$ and
$E(\Gamma_2)=\{\{1,2\},\{2,3\},\{2,4\}\}$, so that
$E(G_I)=\{\{1,2\},\{2,3\},\{2,4\},\{3,4\}\}$. This graph has the
spanning trees $\Gamma_1, \Gamma_2$ and $\Gamma_3$ with
$E(\Gamma_3)=\{\{1,2\},\{2,4\},\{3,4\}\}$. If $\Gamma_3$ would be a
relation tree of $I$, then
\[
A=\begin{pmatrix}
 -x_1 & x_4 & 0 & 0 \\
0 & -x_2x_3 & 0 & x_6 \\
0 & 0 & -x_3x_5 & x_6
\end{pmatrix}.
\]
would have to be a Hilbert--Burch matrix of $I$, which is not the case since the ideal of maximal minors of $A$ is the ideal $x_3I$.

However we have

\begin{Theorem}
\label{linear} Let $I$  be Cohen--Macaulay monomial ideal of
codimension $2$ with linear resolution. Then ${\mathcal
T}(I)={\mathcal T}(G(I))$.
\end{Theorem}

\begin{proof} Since $I$ has a linear resolution, it follows that all Hilbert--Burch matrices of $I$ are matrices with  linear entries. Let $L=\{r_1,\ldots, r_k\}$ be the set of linear Taylor relations.
We may assume that $r_1,\ldots, r_m$ are the rows of a
Hilbert--Burch matrix of $I$, in other words, that $r_1,\ldots, r_m$
is a basis of the first syzygy module $U$ of $I$.

We first claim that $r_{i_1},\ldots, r_{i_m}\in L$ is basis of $U$
if and only if the relations $r_{i_1},\ldots, r_{i_m}$ are
$K$-linear independent. Obviously, the relations must be $K$-linear
independent in order to form  a basis of the free $S$-module $U$.
Conversely, assume that $r_{i_1},\ldots, r_{i_m}$ are $K$-linear
independent. Since each $r_{i_j}$ belongs to $U$ we can write
\[
r_{i_j}=f_{1j}r_1+f_{2j}r_2+\ldots +f_{mj}r_m \quad \text{with} \quad f_{lj}\in S.
\]
The presentation can be chosen such that all $f_{lj}$ are
homogeneous and such that $\deg f_{lj}r_l=\deg r_{i_j}=1$ for all
$l$ and $j$. In other words, $\deg f_{lj}=0$ for all $l$ and $j$.
Therefore the $m\times m$-matrix $F=(f_{lj})$ is a matrix with
coefficients in $K$. Since, by assumption the relations
$r_{i_1},\ldots, r_{i_m}$ are $K$-linear independent, it follows
that $F$ is invertible. This implies that the relations
$r_1,\ldots,r_m$ are linear combinations of the relations
$r_{i_1},\ldots, r_{i_m}$. Therefore these relations generate $U$ as
well, and in fact form a basis of $U$, since $U$ is free of rank
$m$.

Our considerations so far have shown, that the set of Hilbert--Burch
matrices of $I$ correspond bijectively to the maximal $K$-linear
subsets of $L$. Each $r_i\in L$ is a row vector with exactly two
non-zero entries. We attach to $r_i$ the edge $e_i=\{k,l\}$, if the
two non-zero entries of $r_i$ are at position $k$ and $l$, and claim
that
\[
E(G(I))=\{e_1,\ldots,e_k\}.
\]
Indeed, according to the definition of $G(I)$ an edge $e$ belongs to
$E(G(I))$, if there exists a relation tree $T$ of $I$ with $e\in
E(T)$. This is equivalent to say that there exist linearly
independent $r_{i_1},\ldots, r_{i_m}\in L$ such that $e=e_{i_j}$ for
some $j$. Now choose  $e_i\in \{e_1,\ldots,e_k\}$. Then $r_i$ can be
completed to maximal set $\{r_i, r_{i_2},\ldots,r_{im}\}$ of
$K$-linear elements in $L$. This shows that $e_i\in E(G(I))$ for
$i=1,\ldots,k$, so that $\{e_1,\ldots,e_k\}\subset E(G(I))$. The
other inclusion is trivially true.

In order to complete the proof of the theorem we need to show that
each spanning tree $T$ of $G(I)$ is a relation tree of $I$. Let
$e_{i_1},\ldots,e_{i_m}$ be the edges of the tree. To prove  that
$T$ is a relation tree amounts to show the relations
$r_{i_1},\ldots, r_{i_m}$ are $K$-linearly independent.

A free vertex of $T$ is a vertex which belongs to exactly one edge.
Since $T$ is a tree, it has at least one free vertex. Say, $1$ is
this vertex and  $e_{i_1}$ is the edge to which the free vertex $1$
belongs.  Removing the edge $e_{i_1}$ from $T$ we obtain a  tree
$T'$  on the vertex set $\{2,3,\ldots,m+1\}$. After renumbering the
vertices and edges if necessary, we may assume that $2$ is a free
vertex of $T'$ and $e_{i_2}$ the edge to which $2$ belongs.
Proceeding in this way we get, after a suitable renumbering of the
vertices and edges of $T$, a free vertex ordering of the edges, that
is, for all $j=1,\ldots,r$ the edges
$e_{i_j},e_{i_{j+1}},\ldots,e_{i_m}$ is the set of edges of a tree
for which $j$ is a free vertex belonging to $e_{i_j}$. Since
renumbering of vertices and of edges of $T$ means for the
corresponding matrix of relations simply permutation of the rows and
columns, the rank of relation matrix is unchanged. However in this
new ordering, if we skip the last column of  the $m\times m+1$
relation matrix we obtain an upper triangular $m\times m$ matrix
with non-zero entries on the diagonal. This shows that the relations
$r_{i_1},\ldots, r_{i_m}$ are $K$-linearly independent, as desired.
\end{proof}

Finally we will describe all the possible Taylor graphs of a
Cohen--Macaulay monomial ideal of codimension $2$  with linear
resolution. Then, together with Theorem~\ref{linear}, we have a
complete description of all possible relation trees for such ideals.

Let $G$ be finite connected simple graph on the vertex set $[n]$. Recall that a subset $C$ of $[n]$ is called a {\em clique} of $G$ if for
all $i$ and $j$ belonging to $C$ with $i \neq j$ one has $\{ i , j
\} \in E(G)$. The set of all cliques  $\Delta(G)$ is a simplicial complex, called the {\em clique complex} of $G$.

\begin{Theorem}
\label{chordal} Let $G$ be finite connected simple graph.
Then the following are equivalent:
\begin{enumerate}
\item[{\em (a)}] $G$ is a Taylor graph of a Cohen--Macaulay monomial ideal of codimension $2$  with linear
 resolution.

\item[{\em (b)}] $G$ is a chordal graph with the property that any two distinct maximal cliques have at most one vertex in common.
\end{enumerate}
\end{Theorem}

\begin{proof} (a)\implies (b): Let $I$ be generated by $m$ monomials and $G=G(I)$, and let $C$  be a cycle of $G$. We first show that the restriction $G^\prime$ of $G$ to $C$ is a complete graph, that is, we show that for any two distinct vertices $i,j\in C$ it follows that $\{i,j\}\in E(G)$. In particular, this will imply that $G$ is chordal.

For simplicity we may assume that $E(C)=\{e_1,\ldots, e_k\}$ with
$k\geq 3$ and  $e_i=\{i,i+1\}$ for $i=1,\ldots, k-1$ and
$e_k=\{k,1\}$. Let $r_1,\ldots,r_k$ be the corresponding relations.
Let $\epsilon_i\in K^{m-1}$, $i=1,\ldots m-1$  be the canonical
basis vectors of $K^{m-1}$. Then
$r_i=-a_i\epsilon_i+b_i\epsilon_{i+1}$ for $i=1,\ldots,k-1$ and
$r_k=-b_k\epsilon_1+a_k\epsilon_{k}$, where $a_i$ and $b_i$ belong
to $\{x_1,\ldots,x_n\}$. Assume that $r_1,\ldots, r_k$ are
$K$-linearly independent. Then $r_1,\ldots,r_k$ can be completed to
$K$-basis $r_1,\ldots,r_m$ of $L$. (Here we use the notation
introduced in the proof of Theorem~\ref{linear}.) Let $\Gamma$ be
the tree corresponding to $r_1,\ldots,r_m$. Then $C$ is a subgraph
of $\Gamma$, which is a contradiction. Thus we see that the
relations $r_1,\ldots,r_k$ are $K$-linearly dependent which implies
at once that $a_1=b_k$ and $a_i=b_{i-1}$ for $i=2,\ldots,k$. Hence
we have $r_1+\cdots+ r_i=-a_1\epsilon_1+b_{i}\epsilon_{i+1}$ for
$i=1,\ldots,k-1$.  This implies that $\{1,i\}$ is an edge of $G$ for
$i=2,\ldots k$. By symmetry, also the other edges $\{i,j\}$ with
$2\leq i<j\leq k$ belong to $G$.

Now let $G_1$ and $G_2$ be two distinct maximal  cliques of $G$, and
assume that they have two vertices in common, say,  the vertices $i$
and $j$. Let $k\in G_1\setminus \{i,j\}$ and $l\in G_2\setminus
\{i,j\}$. Then the graph $C$ with edges $\{i,k\}, \{k,j\}, \{j,l\},
\{l,i\}$ is a cycle in $G$. Therefore, by what we have shown, it
follows that $\{k,\l\}$ is an edge of $G$. Thus for any two
vertices  $k, l\in V(G_1)\union V(G_2)$ it follows that $\{k,l\}\in
E(G)$, contradicting the fact that $G_1$ and $G_2$ are distinct
maximal cliques of $G$.

(b)\implies (a): Let $C_1,\ldots, C_r$ be the maximal cliques of the
chordal graph $G$, and let $\Delta(G)$ be the clique complex of $G$.
Then the $C_i$ are the facets of  $\Delta(G)$.  One version of
Dirac's theorem \cite{D}  says that $\Delta(G)$ is a quasi-forest,
see \cite{HHZ}. This means, that there is an order of the facets,
say,  $C_1,C_2,\ldots, C_r$ such that for each $i$ there is a $j<i$
with the property that $C_k\sect C_i\subset C_j\sect C_i$ for all
$k<i$. Given this order, then our hypothesis (b) implies that for
each $i=2,\ldots,r$ there exists a vertex $k_i\in C_i$ such
$C_i\sect C_{i-1}=\{k_i\}$ and $C_i\sect C_j=\{k_i\}$ for all $j<i$
with $C_i\sect C_j\neq \emptyset$. The following example illustrates
the situation. Let $G$ be the graph on the vertex set $[7]$ with
edges $\{1,2\}, \{1,3\}, \{2,3\}, \{3,4\}, \{3,5\}, \{4,5\},
\{5,6\}, \{5,7\}$. 




Then $G$ is a connected simple graph satisfying the condition in
(b). The  maximal cliques of $G$ ordered as above are
$C_1=\{1,2,3\}$, $C_2=\{3,4,5\}$, $C_3=\{5,6\}$ and $C_4=\{5,7\}$
and intersection vertices are $k_2=3$, $k_3=5$ and $k_4=5$.

After having fixed the order of the cliques, we may assume that the
vertices of $G$ are labeled as follows: if $|C_1\union \cdots\union
C_i|=s_i$, then $C_1\union \cdots\union C_i=\{1,2,\ldots,s_i\}$. In
other words, $C_1=\{1,\ldots,s_1\}$ and $C_i\setminus
\{k_i\}=\{s_{i-1}+1,\ldots, s_i\}$ for $i>1$. The vertices on the
graph in Figure~1 are labeled in this way. Now we let $\Gamma\subset
G$ be the spanning tree of $G$ whose edges are $\{j,k_2\}$ with
$j\in C_1$ and $j\neq k_2$, and for $i=1,\ldots,r$ the edges
$\{j,k_i\}$ with $j\in C_i$ and $j\neq k_i$. In our example the
edges of $\Gamma$ are $\{1,3\}$, $\{2,3\}$, $\{3,4\}$,$\{3,5\}$,
$\{5,6\}$ and $\{5,7\}$.




Let $m+1=s_r$. Then $m+1$ is the number of vertices of $G$. We now
assign to $\Gamma$ the following $m\times m+1$-matrix $A$ whose rows
$r_e$ correspond to the edges $e$ of $\Gamma$ as follows: we set
$r_e=-x_{1j}\epsilon_j+x_{1k_2}\epsilon_{k_2}$ for $e=\{j,k_2\}$ and
$j\in C_1$ with $j\neq k_2$, and we set
$r_e=-x_{ij}\epsilon_j+x_{ik_i}\epsilon_{k_i}$ for $e=\{j,k_i\}$ and
$j\in C_i$ with $j\neq k_i$ and $i>1$. Here $\epsilon_i$ denotes the
$i$th canonical unit vector in $\RR^{m+1}$.

The rows $r_e$  can be naturally ordered according to the size of $j$ in the edge $e=\{j,k_i\}$. Thus in our example we obtain the matrix
\[
\begin{pmatrix}
-x_{11} & 0 & x_{13} & 0 & 0 & 0 & 0\\
0 & -x_{12} & x_{13} & 0 & 0 & 0 & 0\\
0 & 0 & x_{23} & -x_{24} & 0 & 0 & 0\\
0 & 0 & x_{23} & 0 & -x_{25} & 0 & 0\\
0 & 0 & 0 & 0 & x_{35} & -x_{36} & 0\\
0 & 0 & 0 & 0 & x_{45} & 0 & -x_{47}\\
\end{pmatrix}
\]
Our next goal is to show that our matrix $A$ is a Hilbert--Burch
matrix. We apply Theorem~\ref{juergen}. Tor each edge $\{i,j\}\in
\Gamma$ the monomials $u_{ij}$ and $u_{ji}$ are, according to the
choice of $A$, the following:
\[
u_{jk_2}=-x_{1j},\quad u_{k_2j}=x_{1k_2}\quad \text{for}\quad  j<k_2,
\]
and for $i=2,\ldots, r$
\[
u_{k_ij}=x_{ik_i},\quad u_{jk_i}=-x_{ij} \quad \text{for}\quad k_i<j,\quad j\in C_i.
\]
According to Theorem~\ref{juergen}(b) we have to show that
$\gcd(u_{i b(i,j)},u_{j e(i,j)})=1$ for all $i<j$. Assume first that
$i,j\not\in\{k_2,\ldots,k_r\}$. Then $u_{i b(i,j)}=-x_{ti}$ for
$i\in C_t$ and $x_{j e(i,j)}=-x{sj}$ for$j\in C_s$. Thus in this
case $\gcd(u_{i b(i,j)},u_{j e(i,j)})=1$. In the second case let
$i\not\in\{k_2,\ldots,k_r\}$ and $j\in\{k_2,\ldots,k_r\}$, let say
$j=k_s$. Then $b(i,j)=k_t$ for $i\in C_t$ and so
$u_{ib(i,j)}=-x_{ti}$. Suppose $\{i,j\}$ is not an edge then
$e(i,j)=b(j,i)=b(k_s,i)$ is either  $k_{s+1}$  or  $k_{s-1}$.  Then
$u_{je(i,j)}$ is either $-x_{(s)j}$   or  $-x_{(s-1)j}$. On the
other hand,  if $\{i,j\}$ is an edge, then $e(i,j)=i$, and   so
$u_{j e(i,j)}=u_{ji}=u_{k_si}=x_{sj}$. Thus in this case, too,
$\gcd(u_{i b(i,j)},u_{j e(i,j)})=1$. Finally assume that
$i,j\in\{k_2,\ldots,k_r\}$, and let
$i=k_{s_1},k_{s_1+1},\ldots,k_{s_2}=j$ be the path from $i$ to $j$.
Then $b(i,j)=k_{s_1+1}$ and $e(i,j)=k_{s_2-1}$ so
$u_{ib(i,j)}=x_{s_1i}$ and $u_{je(i,j)}=-x_{(s_2-1)j}$. Thus again
in this case we have $\gcd(u_{i b(i,j)},u_{j e(i,j)})=1$. Thus in
all cases $\gcd(u_{i b(i,j)},u_{j e(i,j)})=1$ for all $i<j$, as
desired.

Let $I$ be the codimension $2$ ideal whose relation matrix is $A$,
and let $\{i,j\}$ be any edge of $G$. It remains to be shown that
there exits a relation tree $\Gamma'$ with $\{i,j\}\in E(\Gamma')$.
If $\{i,j\}\in E(\Gamma)$, we are done. Now assume that
$\{i,j\}\not\in E(\Gamma)$. We may assume that $\{i,j\}\in C_t$. Let
$s=t$ if $t>1$, and $s=2$ if $t=1$.  We replace the row
$-x_{tj}\epsilon_j+x_{tk_s}\epsilon_{k_s}$ of $A$ by the difference
of the rows
\[
-x_{ti}\epsilon_i+x_{tj}\epsilon_{j}=(-x_{ti}\epsilon_i+x_{tk_s}\epsilon_{k_s})-
(-x_{tj}\epsilon_j+x_{tk_s}\epsilon_{k_s}),
\]
and leave all the other rows of $A$ unchanged. The new matrix $A'$
is again a relation matrix of $I$ and the tree $\Gamma'$
corresponding to $A'$ is obtained from $\Gamma$ by removing the edge
$\{j,k_s\}$ and adding the edge $\{i,j\}$. This completes the proof
of the theorem.
\end{proof}

\end{document}